\documentclass[reqno]{amsart}

\usepackage{amsmath,amssymb,amsthm,comment,enumerate,mathrsfs}

\usepackage{hyperref}

\usepackage{color}

\numberwithin{equation}{section}

\newtheorem{Thm}{Theorem}[section]
\newtheorem{Prop}[Thm]{Proposition}
\newtheorem{Lem}[Thm]{Lemma}
\newtheorem{Cor}[Thm]{Corollary}
\theoremstyle{definition}
\newtheorem{Rem}[Thm]{Remark}

\newtheorem{Expl}[Thm]{Example}
\theoremstyle{remark}
\newtheorem*{Rmk*}{Remark}
\newtheorem*{Cl*}{Claim}

\newcommand{\R}{\mathbb{R}}

\newcommand{\bary}{\operatorname{bar}}

\newcommand{\eps}{\varepsilon}

\newcommand{\id}{\operatorname{id}}

\renewcommand{\rho}{\varrho}

\newcommand{\sig}{\sigma}

\title
{
Ball Intersection Properties in Metric Spaces
}

\author{
Benjamin Miesch and
Ma\"el Pav\'on}
\address{Department of Mathematics, ETH Z\"urich, 8092 Z\"urich, Switzerland}
\email{benjamin.miesch@math.ethz.ch}
\email{mael.pavon@math.ethz.ch}

\date{\today}

\begin{document}


\begin{abstract}
We show that in complete metric spaces, $4$-hyperconvexity is equivalent to finite hyperconvexity. Moreover, every complete, almost $n$-hyperconvex metric space is $n$-hyperconvex. This generalizes among others results of Lindenstrauss and answers questions of Aronszajn-Panitchpakdi.

Furthermore, we prove local-to-global results for externally and weakly externally hyperconvex subsets of hyperconvex metric spaces and find sufficient conditions in order for those classes of subsets to be convex with respect to a geodesic bicombing.
\end{abstract}

\maketitle


\section{Introduction} \label{Sec:Intro}


Hyperconvexity and related properties lie at the interface of several fields like fixed point theory~\cite{Sin}, mapping extensions~\cite{AroP,EspL}, functional analysis~\cite{Lin,Nac}, geometric group theory~\cite{Lan} or convex geometry~\cite{Han}. Motivated by these applications, we study weak notions of hyperconvexity and afterwards connect them with the theory of convex subsets. Finally, this leads to a Helly-type theorem for weakly externally hyperconvex subsets.

In the first part of this work we generalize results related to extensions of uniformly continuous functions and compact linear operators. In \cite{Lin}, Lindenstrauss characterizes all Banach spaces $B$ with the property that any compact linear operator with target $B$ possesses an "almost" norm preserving extension in pure metric terms, namely as the Banach spaces which are $n$-hyperconvex for every $n$. A counterpart for uniformly continuous maps between metric spaces was later proven by Esp\'inola and L\'opez, see \cite{EspL}. This motivates a closer look on results concerning $n$-hyperconvexity in general metric spaces. Note that the following definition is slightly different from the one given in \cite{AroP}.

All through this paper, $B(x,r)$ denotes the closed ball with center $x$ and radius $r \geq 0$.
Let $A$ be a subset of a metric space $(X,d)$. The subset $A$ is \ldots
	\begin{itemize}
	\item \emph{$n$-hyperconvex} if for every family of $n$ closed balls $\{ B(x_i,r_i) \}_{i=1}^n$ with $x_i \in A$ and $d(x_i,x_j) \leq r_i + r_j$, we have $\bigcap_{i=1}^n B(x_i,r_i) \cap A \neq \emptyset$;
	\item \emph{almost $n$-hyperconvex} if for every family of $n$ closed balls $\{ B(x_i,r_i) \}_{i=1}^n$ with $x_i \in A$ and $d(x_i,x_j) \leq r_i + r_j$, we have $\bigcap_{i=1}^n B(x_i,r_i + \epsilon) \cap A \neq \emptyset$, for every $\epsilon>0$;
	\item \emph{externally $n$-hyperconvex} in $X$ if for every family of $n$ closed balls $\{ B(x_i,r_i) \}_{i=1}^n$ with $x_i \in X$, $ d(x_i,A) \leq r_i$ and $d(x_i,x_j) \leq r_i + r_j$, we have $\bigcap_{i=1}^n B(x_i,r_i) \cap A \neq \emptyset$;
	\item \emph{weakly externally $n$-hyperconvex} in $X$ if for every $x \in X$ the set $A$ is externally $n$-hyperconvex in $A \cup \{x\}$.
	\end{itemize}
	
Accordingly, we call $A$ \emph{hyperconvex}, \emph{almost hyperconvex}, \emph{externally hyperconvex} or \emph{weakly externally hyperconvex} if the corresponding property holds for arbitrary families of closed balls.

The following two theorems supplement results proven for Banach spaces by Lindenstrauss, see \cite[Lemma~4.2]{Lin} and \cite[Lemma~2.13]{BenL}, and hence completely answer Problem~1 and Problem~4 raised by Aronszajn and Panitchpakdi in \cite{AroP}.

\begin{Thm}\label{Thm:AlmostAndComplete}
	Let $X$ be a complete, almost $n$-hyperconvex metric space for $n \geq 3$. Then $X$ is $n$-hyperconvex.
\end{Thm}

This implies for instance that the metric completion of an $n$-hyperconvex metric space is $n$-hyperconvex as well. Note that there are complete metric spaces which are almost 2-hyperconvex but not 2-hyperconvex, cf. \cite{AroP}.

\begin{Thm}\label{Thm:FourToFinite}
Let $X$ be a complete metric space and let $A \subset X$ be an arbitrarily chosen non-empty subset. Then, the following hold:
\begin{enumerate}[(i)]
\item $X$ is 4-hyperconvex if and only if $X$ is $n$-hyperconvex for every $n$.
\item $A$ is externally 4-hyperconvex in $X$ if and only if $A$ is externally $n$-hyperconvex in $X$ for every $n$.
\item $A$ is weakly externally 4-hyperconvex in $X$ if and only if $A$ is weakly externally $n$-hyperconvex in $X$ for every $n$.
\end{enumerate}
\end{Thm}

Observe that this is the best we can hope for, since there are metric spaces which are $3$-hyperconvex but not $4$-hyperconvex, e.g. $l_1^3$, and there is a subset $A$ of $l_\infty(\mathbb{N})$ which is externally $n$-hyperconvex for every $n$, but fails to be hyperconvex, see Example~\ref{Expl:externally n-hyperconvex but not hyperconvex}.

We use a completely new approach to this problem, establishing the finite binary intersection property for externally 2-hyperconvex subsets. In contrast to our methods, the proofs of Lindenstrauss are based on the existence of barycenters and therefore they can easily be adapted to spaces with a geodesic bicombing (see the definition below). In the appendix we will give a proof of the more general theorem of Lindenstrauss on the $(n,k)$-intersection property in this setting, compare \cite[Theorem~4.1]{Lin}.

The second part of this work relates weak notions of convexity of subsets of a metric space to the notions of (weak) external hyperconvexity. The first challenge at this point consists in defining convexity in general hyperconvex metric spaces. In this context, the idea of geodesic bicombings as introduced by Lang in \cite{Lan} arises from the lack of unique geodesics.

A \emph{geodesic bicombing} is a map
\[
	\sigma \colon X \times X \times [0,1] \to X
\]
such that for each pair $(x,y) \in X \times X$, the map $\sigma_{xy} := \sigma(x,y,\cdot) \colon [0,1] \to X$ is a constant speed geodesic from $x$ to $y$, i.e. $\sigma_{xy}(0)=x$, $\sigma_{xy}(1)=y$ and for all $s,t \in [0,1]$, one has $d(\sigma_{xy}(s),\sigma_{xy}(t))=|s-t| d(x,y)$.
Moreover, we assume that the choice of geodesics is symmetric, i.e. for all $x,y \in X$ and $t \in [0,1]$ one has 
\[
	\sigma_{yx}(t) = \sigma_{xy}(1-t),
\]
and that the geodesics fulfill the following weak convexity assumption:
For all $x,y,x',y' \in X$ and $t\in [0,1]$, one has
\[
	d(\sigma_{xy}(t),\sigma_{x'y'}(t)) \leq (1-t)d(x,x')+t d(y,y').
\]

Recall that every hyperconvex metric space admits a geodesic bicombing. Furthermore, if the metric space $X$ is \emph{proper}, i.e. closed balls in $X$ are compact, and admits a geodesic bicombing, then the space $X$ also possesses a \emph{convex} geodesic bicombing $\sigma$, that is a geodesic bicombing with the property that for all $x,y,x',y' \in X$ the function $t \mapsto d(\sigma_{xy}(t),\sigma_{x'y'}(t))$ is convex, cf. \cite[Theorem~1.1]{DesL}. Geodesics of a convex geodesic bicombing especially are \emph{straight} curves, i.e. $t \mapsto d(z,\sigma_{xy}(t))$ is convex for every $z \in X$, which leads to uniqueness results.
For instance, in normed vector spaces or in metric spaces with finite combinatorial dimension, straight curves are unique and therefore such spaces admit at most one convex geodesic bicombing \cite[Theorem~3.3, Proposition~4.3]{DesL}. This unique convex geodesic bicombing must be \emph{consistent}, that is, for every $x,y \in X$ and $0 \leq s \leq t \leq 1$, we have $\sigma_{x'y'}([0,1]) \subset \sigma_{xy}([0,1])$, where $x' = \sigma_{xy}(s)$, $y' = \sigma_{xy}(t)$. Note that every consistent geodesic bicombing is convex, but the converse is not true in general.
For an extensive study of existence and uniqueness of geodesic bicombings, we refer to \cite{DesL,BasM}.

Given a metric space $X$ with a geodesic bicombing $\sigma$, we say that a subset $A$ of $X$ is \emph{$\sigma$-convex} if for every $x,y \in A$ we have $\sigma_{xy}([0,1]) \subset A$. If $X$ is a normed vector space with the usual linear bicombing, this notion coincides with the ordinary convexity.

We will show that, under the appropriate assumptions on the bicombing $\sigma$, (weakly) externally hyperconvex subsets are $\sigma$-convex. We then combine these convexity results with local-to-global properties of (weakly) externally hyperconvex subsets with a geodesic bicombing. A subset $A$ of a metric space $X$ is \emph{uniformly locally (weakly) externally hyperconvex} if there is some $r>0$ such that for every $x$ in $A$ the set $A \cap B(x,r)$ is (weakly) externally hyperconvex in $B(x,r)$.

With these definitions at hand, we can state our third main result.

\begin{Thm}\label{Thm:BicombingTheorem}
Let $X$ be a hyperconvex metric space, let $A \subset X$ be any subset and let $\sigma$ denote a convex geodesic bicombing on $X$.
\begin{enumerate}[(I)]
\item The following are equivalent:
\begin{enumerate}[(i)]
\item $A$ is externally hyperconvex in $X$.
\item $A$ is $\sigma$-convex and uniformly locally externally hyperconvex.
\end{enumerate}

\item If straight lines in $X$ are unique, the following are equivalent:
\begin{enumerate}[(i)]
\item $A$ is weakly externally hyperconvex and possesses a consistent geodesic bicombing.
\item $A$ is $\sigma$-convex and uniformly locally weakly externally hyperconvex.
\end{enumerate}
\end{enumerate}
\end{Thm}

Theorem~\ref{Thm:BicombingTheorem} shows that there are substantial differences between hyperconvex subsets and externally or weakly externally subsets of a metric space, since hyperconvex subsets are \emph{not} convex in general.

Observe that Theorems~\ref{Thm:FourToFinite} and \ref{Thm:BicombingTheorem} can also be combined. Let for example $X$ be a proper $4$-hyperconvex metric space and let $A \subset X$ be any subset. Then, $X$ is hyperconvex by Theorem~\ref{Thm:FourToFinite} and $X$ admits a convex geodesic bicombing by \cite[Proposition~3.8]{Lan} and \cite[Theorem~1.1]{DesL}. Therefore, Theorem~\ref{Thm:BicombingTheorem} applies.

As another application of Theorems~\ref{Thm:FourToFinite} and \ref{Thm:BicombingTheorem}, consider now the simple case of the $n$-dimensional normed space $l^n_{\infty}:= (\R^n, \| \cdot\|_{\infty})$, where $\| \cdot\|_{\infty}$ denotes the maximum norm. If $A$ is a hyperconvex subset of $l^n_{\infty}$, then $A$ does not need to be convex. However, as soon as $A$ is weakly externally 4-hyperconvex in $l^n_{\infty}$, it follows by Theorem~\ref{Thm:FourToFinite} that $A$ is weakly externally hyperconvex and thus $A$ must be convex by Theorem~\ref{Thm:BicombingTheorem}. 

This leads to the following Helly-type result. A family of sets is a \emph{Helly family of order $k$} if every finite subfamily such that every $k$-fold intersection is non-empty has non-empty total intersection. For instance, if $X$ is a 4-hyperconvex metric space, then the set $\mathcal{E}_2(X)$ of externally 2-hyperconvex subsets is a Helly family of order $2$.

\begin{Thm}\label{Thm:HellyWEH}
The collection $\mathcal{W}(l^n_{\infty})$ of weakly externally hyperconvex subsets of $l_\infty^n$ is a Helly family of order $n+1$. Moreover, the order is optimal.
\end{Thm}

Finally, we give a further characterization of externally hyperconvex subsets in terms of retractions.


\section{Finite hyperconvexity} \label{Sec:Finite}


In this section we will have a closer look at $n$-hyperconvex metric spaces and will eventually prove Theorems~\ref{Thm:AlmostAndComplete} and \ref{Thm:FourToFinite}.

We denote the collection of all $n$-hyperconvex, externally $n$-hyperconvex and weakly externally $n$-hyperconvex subsets of $X$ by $\mathcal{H}_n(X)$, $\mathcal{E}_n(X)$ and $\mathcal{W}_n(X)$, respectively. It holds that $\mathcal{E}_n(X) \subset \mathcal{W}_n(X) \subset \mathcal{H}_n(X)$ and $\mathcal{H}_{n+1}(X) \subset \mathcal{H}_n(X)$.

\begin{Rmk*}
	Obviously, every metric space is $1$-hyperconvex. A subset $A$ of $X$ is (weakly) externally $1$-hyperconvex if and only if it is \emph{proximinal}, i.e. for every $x \in X$ there is some $a \in A$ with $d(x,a) = d(x,A)$. Recall that proximinal subsets are closed.
	Furthermore, a metric space $X$ is $2$-hyperconvex if and only if it is \emph{metrically convex}, i.e. for every $x_0,x_1 \in X$ and every $t \in [0,1]$ there is some $x_t \in X$ with $d(x_0,x_t) = t d(x_0,x_1)$ and $d(x_t,x_1) = (1-t) d(x_0,x_1)$, compare \cite[Definition~1.3]{BenL}.
\end{Rmk*}

	For two points $x,y$ in a metric space $(X,d)$, the set $$I(x,y) := \{ z \in X : d(x,y) = d(x,z) + d(z,y) \}$$ is called the \emph{metric interval} between $x$ and $y$.
	
	A metric space $(X,d)$ is called \emph{modular} if for all $x,y,z \in X$ the median set $$M(x,y,z) := I(x,y) \cap I(y,z) \cap I(z,x)$$ is non-empty.

\begin{Prop}\label{Prop:3-hyperconvex}
	A metric space $X$ is 3-hyperconvex if and only if it is metrically convex and modular.
\end{Prop}

\begin{proof}
	First assume that $X$ is 3-hyperconvex. Recall the Gromov product $(y|z)_x := \frac{1}{2}(d(x,y) + d(x,z) - d(y,z))$ for $x,y,z \in X$. Let $x_1,x_2,x_3 \in X$ and define $r_i := (x_j | x_k)_i$ for $\{i,j,k\} = \{1,2,3\}$. Then we have $$M(x_1,x_2,x_3) = \bigcap_{i=1}^3 B(x_i,r_i) \neq \emptyset.$$
	
	Conversely, let $X$ be metrically convex and modular. Let $x_1,x_2,x_3 \in X$ and $r_1,r_2,r_3 \in \mathbb{R}$ with $d(x_i,x_j) \leq r_i + r_j$. If $M(x_1,x_2,x_3) \cap \bigcap_{i=1}^3 B(x_i,r_i) = \emptyset$, we might assume without loss of generality that $r_3 < (x_1|x_2)_{x_3}$. Take $m \in M(x_1,x_2,x_3)$ and define $r_m = \min \{ r_1 - d(x_1,m), r_2 - d(x_2,m)\}$. Since $X$ is metrically convex we get $$\emptyset \neq B(m,r_m) \cap B(x_3,r_3) \subset \bigcap_{i=1}^3 B(x_i,r_i).$$	
\end{proof}

	As a first step, we now turn our attention to almost $n$-hyperconvex metric spaces and its subsets.
	A subset $A$ of a metric space $X$ is called \emph{almost externally $n$-hyperconvex} in $X$ if for every family $\{ B(x_i,r_i) \}_{i=1}^n$ of closed balls with $d(x_i,x_j) \leq r_i + r_j$ and $d(x_i,A) \leq r_i$ we have $$\bigcap_{i=1}^n B(x_i,r_i + \epsilon) \cap A \neq \emptyset,$$ for every $\epsilon > 0$. We denote the set of all closed, almost externally $n$-hyperconvex subsets of $X$ by $\mathcal{\tilde{E}}_n(X)$.
	
	Similarly, we say that a subset $A$ of a metric space $X$ is \emph{almost weakly externally $n$-hyperconvex} in $X$ if it is almost externally $n$-hyperconvex in $A \cup \{x\}$ for every $x \in X$. We denote all such subsets which are closed by $\mathcal{\tilde{W}}_n(X)$.

Aronszajn and Panitchpakdi already showed that every complete, almost $(n+1)$-hyperconvex metric space is $n$-hyperconvex \cite[Theorem 3.4]{AroP}. This result easily extends to almost (weakly) externally hyperconvex subsets.

\begin{Lem}\label{Lem:almost externally (n+1)-hyperconvex}
	Let $X$ be a complete metric space. Then we have
\begin{enumerate}[(i)]
	\item $A \in \mathcal{\tilde{E}}_{n+1}(X) \Rightarrow A \in \mathcal{E}_n(X)$,
	\item $A \in \mathcal{\tilde{W}}_{n+1}(X) \Rightarrow A \in \mathcal{W}_n(X)$, and
	\item if $X$ is almost $(n+1)$-hyperconvex, then $X$ is $n$-hyperconvex.
\end{enumerate} 
\end{Lem}

\begin{proof}
	We will prove (i). The statements (ii) and (iii) then easily follow. Let $\{ B(x_i,r_i) \}_{i=1}^n$ be a family of closed balls with $d(x_i,x_j) \leq r_i + r_j$ and $d(x_i,A) \leq r_i$. Starting with $y_1 \in \bigcap_{i=1}^n B(x_i, r_i+ \frac{1}{2}) \cap A$, we construct inductively a sequence $(y_k)_k$ by choosing
$$y_{k+1} \in B(y_k,\tfrac{1}{2^k}+\tfrac{1}{2^{k+1}}) \cap \bigcap_{i=1}^n B(x_i, r_i+ \tfrac{1}{2^{k+1}}) \cap A.$$
Since $d(y_k,y_{k+1}) \leq \frac{1}{2^k}+\frac{1}{2^{k+1}}$ this sequence is Cauchy and converges to some $y \in A$ with $$d(x_i,y) = \lim_{k \to \infty} d(x_i,y_k) \leq \lim_{k \to \infty} r_i + \tfrac{1}{2^k} = r_i.$$
	That is $y \in \bigcap_{i=1}^n B(x_i,r_i) \cap A \neq \emptyset$.
\end{proof}

\begin{Lem}\label{Lem:externally (n-k)-hyperconvex}
	Let $X$ be a complete metric space, $A \in \mathcal{\tilde{E}}_n(X)$ and $k \leq n-2$. Then for every family $\{ B(x_i,r_i) \}_{i=1}^k$ with $d(x_i,x_j) \leq r_i+r_j$ and $d(x_i,A) \leq r_i$, we have $A' := A \cap \bigcap_{i=1}^k B(x_i,r_i) \in \mathcal{\tilde{E}}_{n-k}(X)$.
\end{Lem}

\begin{proof}
	Let $\{ B(x_i,r_i) \}_{i=k+1}^n$ be a family of closed balls with $d(x_i,x_j) \leq r_i + r_j$ and $d(x_i,A') \leq r_i$. Fix some $\epsilon > 0$. Then there is some $y \in A \cap \bigcap_{i=1}^n B(x_i,r_i + \tfrac{\epsilon}{2})$ and by Lemma~\ref{Lem:almost externally (n+1)-hyperconvex}(i), we get $$\emptyset \neq A \cap \bigcap_{i=1}^k B(x_i,r_i) \cap B(y,\tfrac{\epsilon}{2}) \subset A' \cap \bigcap_{i=k+1}^n B(x_i,r_i + \epsilon).$$
\end{proof}

\begin{Cor}\label{Cor:externally (n-k)-hyperconvex}
	Let $X$ be a complete, almost $n$-hyperconvex metric space and let $k \leq n-2$. Then for every family $\{ B(x_i,r_i) \}_{i=1}^k$ with $d(x_i,x_j) \leq r_i+r_j$, we have $B := \bigcap_{i=1}^n B(x_i,r_i) \in \mathcal{\tilde{E}}_{n-k}(X)$.
\end{Cor}

By investigating the proofs of Lemma~4.3 and Lemma~4.4 in \cite{Mie}, we see that the requirements there can be weakened as follows.

\begin{Lem}\label{Lem:almost 3-hyperconvex}
	Let $X$ be a complete, almost 3-hyperconvex metric space. Let $A, A' \in \mathcal{\tilde{E}}_2(X)$ with $y \in A\cap A'\neq \emptyset$ and $x \in X$ with $d(x,A),d(x,A') \leq r$. Denote $d:=d(x,y)$ and $s:=d-r$.
	Then for every $\epsilon > 0$ and every $\delta > 0$, we have $$A\cap A' \cap B(x,r + \delta) \cap B(y,s+\epsilon) \neq \emptyset,$$ given $s\geq 0$.
	In any case, the intersection $A\cap A' \cap B(x,r+ \delta)$ is non-empty.
\end{Lem}

\begin{proof}
	We will split the proof into three steps.
\\ \mbox{} \\
\noindent \textbf{Step I.} \emph{For all $\epsilon > 0$ and $\delta > 0$, there are $a \in C_\delta := B(x,r+\delta) \cap A$ and $a'\in C_\delta' := B(x,r+\delta) \cap A'$ such that $d(a,a') \leq \epsilon$ and $d(y,a) \leq s + \epsilon$.}
\\

	Let $0 \leq \tilde{\epsilon} \leq \frac{\epsilon}{2}$, $n_0 := \left\lfloor \frac{s}{\tilde{\epsilon}} \right\rfloor$ and $0 < \delta \leq \frac{\tilde{\epsilon}}{n_0+1}$, we start by choosing
	\begin{align*}
		a_1 &\in B(y, \tilde{\epsilon} + \tfrac{\delta}{2}) \cap B(x,d-\tilde{\epsilon} + \tfrac{\delta}{2}) \cap A, \\
		a_2 &\in B(a_1, \tilde{\epsilon} + \tfrac{\delta}{2} + \tfrac{\delta}{4}) \cap B(x,d-2\tilde{\epsilon} + \tfrac{\delta}{2} + \tfrac{\delta}{4} ) \cap A'. 
	\end{align*}
	Then, as long as $n \leq n_0$, we can inductively pick $$a_n \in B(a_{n-1},\tilde{\epsilon} + \Delta_n) \cap B(x,d-n\tilde{\epsilon} + \Delta_n ) \cap A,$$ if $n$ is odd and $$a_n \in B(a_{n-1},\tilde{\epsilon} + \Delta_n) \cap B(x,d-n\tilde{\epsilon} + \Delta_n ) \cap A',$$ if $n$ is even, where $\Delta_n := \sum_{i=1}^n \frac{\delta}{2^i} \leq \delta$. Observe that we have $d(y,a_n) \leq n (\tilde{\epsilon} + \delta)$.
	 Finally, assuming without loss of generality $a_{n_0} \in A'$ (otherwise interchange $A$ and $A'$), there are $$a \in B(a_{n_0},\tilde{\epsilon} + \Delta_{n_0+1}) \cap B(x,r + \Delta_{n_0+1} ) \cap A \subset C_\delta$$ and $$a' \in B(a,\tilde{\epsilon} + \Delta_{n_0+2}) \cap B(x,r + \Delta_{n_0+2} ) \cap A' \subset C_\delta'.$$ We have $$d(a,a') \leq \tilde{\epsilon} + \delta \leq \epsilon$$ and $$d(a,y)\leq (n_0+1) ( \tilde{\epsilon} + \delta) \leq s + \tilde{\epsilon} + (n_0+1) \delta \leq s + \epsilon.$$ This finishes Step I.
\\ \mbox{} \\
\noindent \textbf{Step II.} \emph{For all $\epsilon>0$, $\delta>0$ and all $n \geq 1$, there are $a_n \in A$ and $a_n' \in A'$ with 
\begin{itemize}
	\item $d(a_n,a_n') \leq \frac{\epsilon}{2^n}$,
	\item $d(a_{n-1},a_n), d(a_{n-1}',a_n') \leq \frac{2\epsilon+\delta}{2^n}$,
	\item $d(x,a_n),d(x,a_n') \leq r + \Delta_n$, where $\Delta_n := \sum_{i=1}^n \frac{\delta}{2^i}$, and
	\item $d(y,a_n),d(y,a_n') \leq s + \mathrm{E}_n$, where $\mathrm{E}_n := \sum_{i=1}^n \frac{\epsilon}{2^i}$.
\end{itemize}}
\mbox{} \\
\indent	We start by choosing
\begin{align*}
	a_1 \in A \cap B(x,r+\tfrac{\delta}{2}), \\
	a_1' \in A' \cap B(x,r+\tfrac{\delta}{2}),
\end{align*}
with $d(a_1,a_1') \leq \frac{\epsilon}{2}$ and $d(y,a_1) \leq s + \frac{\epsilon}{2}$ according to Step I.

We then continue inductively as follows. First, since $X$ is almost 3-hyperconvex, we can pick some
\begin{align*}
	x_n \in B(a_{n-1},\tfrac{\epsilon}{2^n} + \tfrac{\delta}{2^{n+2}}) \cap B(a_{n-1}',\tfrac{\epsilon}{2^n} + \tfrac{\delta}{2^{n+2}}) \cap B(x, r + \Delta_{n-1} - \tfrac{\epsilon}{2^n} + \tfrac{\delta}{2^{n+2}}).
\end{align*}
We denote $r_n := \tfrac{\epsilon}{2^n} + \tfrac{\delta}{2^{n+2}}$ and $s_n := s + \mathrm{E}_{n-1}$. Observe that
\begin{align*}
	d(x_n,y) \leq d(x_n,a_{n-1}) + d(a_{n-1},y) \leq r_n + s_n.
\end{align*}
Now, by Step I, there are
\begin{align*}
	a_n &\in A \cap B(x_n, r_n + \tfrac{\delta}{2^{n+1}}), \\
	a_n' &\in A' \cap B(x_n, r_n + \tfrac{\delta}{2^{n+1}}),
\end{align*}
with $d(a_n,a_n') \leq \frac{\epsilon}{2^n}$ and $d(y,a_n) \leq s_n + \frac{\epsilon}{2^n} = s + \mathrm{E}_n$.
Moreover, we have
\begin{align*}
	d(a_{n-1},a_n) &\leq d(a_{n-1},x_n) + d(x_n,a_n) \leq 2 r_n + \tfrac{\delta}{2^{n+1}} = \tfrac{2\epsilon+\delta}{2^n},\\
	d(x,a_n) &\leq d(x,x_n) + d(x_n,a_n) \\
	&\leq r + \Delta_{n-1} - \tfrac{\epsilon}{2^n} + \tfrac{\delta}{2^{n+2}} + r_n + \tfrac{\delta}{2^{n+1}} = r + \Delta_n,
\end{align*}
as desired.
\\ \mbox{} \\
\noindent \textbf{Step III.} Observe that the sequences $(a_n)_n,(a_n')_n$ are Cauchy and since $d(a_n,a_n') \to 0$ they have a common limit point
$$a \in A \cap A' \cap B(y,s+\epsilon) \cap B(x,r + \delta) \neq \emptyset.$$
	This concludes the proof.
\end{proof}

\begin{Lem}\label{Lem:intersection of three almost externally hyperconvex sets}
	Let $X$ be a complete, almost 3-hyperconvex metric space and let $A_0 \in \mathcal{\tilde{E}}_3(X)$ and $A_1,A_2 \in \mathcal{\tilde{E}}_2 (X)$ be pairwise intersecting subsets. Then $A_0 \cap A_1 \cap A_2 \neq \emptyset$.
\end{Lem}

\begin{proof}
	Choose some point $x_0 \in A_1 \cap A_2$ and let $r_0 := d(x_0,A_0)$. By Lemma~\ref{Lem:almost 3-hyperconvex} there is $y_0 \in A_0 \cap A_1 \cap B(x_0,r_0+\frac{r_0}{12})$. Define $A_0' := A_0 \cap B(y_0,\frac{7}{6}r) \in \mathcal{\tilde{E}}_2(X)$. Using again Lemma~\ref{Lem:almost 3-hyperconvex} we have $A_0'\cap A_2 = A_0 \cap A_2 \cap B(y_0,\frac{13}{12}r_0+\frac{r_0}{12}) \neq \emptyset$ and therefore there is some $z_0 \in A_0' \cap A_2 \cap B(x_0,\frac{13}{12}r_0+\frac{r_0}{12}) = A_0 \cap A_2 \cap B(x_0,\frac{7}{6}r_0) \cap B(y_0,\frac{7}{6}r_0)$. Then, since $A_0$ is almost externally 3-hyperconvex, there is some
	$$\bar{x}_0 \in B(x_0,r_0+\tfrac{r_0}{12}) \cap B(y_0, \tfrac{7}{12}r_0+\tfrac{r_0}{12}) \cap B(z_0,\tfrac{7}{12}r_0+\tfrac{r_0}{12})\cap A_0$$
	and using again Lemma~\ref{Lem:almost 3-hyperconvex}, we find
	$$x_1 \in A_1 \cap A_2 \cap B(\bar{x}_0,\tfrac{2}{3}r_0 + \tfrac{r_0}{12})\cap B(x_0,\tfrac{5}{12} r_0 + \tfrac{r_0}{12}).$$
	Note that $d(x_1,A_0) \leq d(x_1,\bar{x}_0) \leq \frac{3}{4}r_0 =: r_1$ and $d(x_0,x_1) \leq \frac{1}{2}r_0$. Proceeding this way, we get some sequence $(x_n)_n \subset A_1 \cap A_2$ with $d(x_n, A_0) \leq \left(\frac{3}{4}\right)^nr_0$ and $d(x_n,x_{n+1}) \leq \frac{1}{2}\left(\frac{3}{4}\right)^nr_0$. Hence $(x_n)_n$ is a Cauchy sequence and therefore  converges to some $x \in A_0 \cap A_1 \cap A_2 \neq \emptyset$, since $A_0$ is closed.
\end{proof}

We are now able to give a proof of Theorem~\ref{Thm:AlmostAndComplete}. In fact, 
we will show the following more general result.

\begin{Prop}\label{Prop:almost n-hyperconvex}
	Let $X$ be a complete metric space. Then for every $n \geq 3$ we have
\begin{enumerate}[(i)]
	\item $A \in \mathcal{\tilde{E}}_n(X) \Rightarrow A \in \mathcal{E}_n(X)$,
	\item $A \in \mathcal{\tilde{W}}_n(X) \Rightarrow A \in \mathcal{W}_n(X)$, and
	\item if $X$ is an almost $n$-hyperconvex metric space then $X$ is $n$-hyperconvex.
\end{enumerate}
\end{Prop}

\begin{proof}
	First note that (ii) and (iii) directly follow from (i) since 
\begin{itemize}
	\item $A \in \mathcal{W}_n(X) \Leftrightarrow A \in \mathcal{E}_n(A \cup \{x\})$ for every $x\in X$, and
	\item $A$ is $n$-hyperconvex $\Leftrightarrow A \in \mathcal{E}_n(A)$.
\end{itemize}	
	  To prove (i), we now consider two cases.
\\ \mbox{} \\
\noindent \emph{\textbf{Case 1.}} Let $A \in \mathcal{\tilde{E}}_3(X)$ and let $\{ B(x_i,r_i) \}_{i=1}^3$ be a family of closed balls with $d(x_i,x_j) \leq r_i + r_j$. By Lemma~\ref{Lem:almost externally (n+1)-hyperconvex} $A$ is externally 2-hyperconvex and hence there is some $y_0 \in A \cap B(x_1,r_1) \cap B(x_2,r_2)$. Set $s := d(y_0,x_3)-r_3$. By Lemma~\ref{Lem:externally (n-k)-hyperconvex} we have $A_i := A \cap B(x_i,r_i) \in \mathcal{\tilde{E}}_2(X)$ and $d(x_3,A_i) \leq r_3$, for $i=1,2$, since it holds that $A \cap B(x_i,r_i) \cap B(x_3,r_3) \neq \emptyset$ as $A \in \mathcal{E}_2(X)$. By Lemma~\ref{Lem:almost 3-hyperconvex} we therefore find inductively
\begin{align*}
	y_1 &\in A_1 \cap A_2 \cap B(x_3,r_3+\tfrac{\epsilon}{2}) \cap B(y_0,s+\epsilon), \\
	y_2 &\in A_1 \cap A_2 \cap B(x_3,r_3+\tfrac{\epsilon}{4}) \cap B(y_1,\tfrac{\epsilon}{2}+\tfrac{\epsilon}{2}), \\
	y_n &\in A_1 \cap A_2 \cap B(x_3,r_3+\tfrac{\epsilon}{2^n}) \cap B(y_{n-1},\tfrac{\epsilon}{2^{n-1}}+\tfrac{\epsilon}{2^{n-1}}).
\end{align*}
	Since $d(y_n,y_{n-1}) \leq \frac{\epsilon}{2^{n-2}}$, this is a Cauchy sequence which converges to some $$y \in A_1 \cap A_2 \cap B(x_3,r_3) = A \cap \bigcap_{i=1}^3 B(x_i,r_i) \neq \emptyset.$$

\noindent \emph{\textbf{Case 2.}} Assume now that $n \geq 4$ and let $\{ B(x_i,r_i) \}_{i=1}^n$ be a family of closed balls with $d(x_i,x_j) \leq r_i + r_j$, $d(x_i,A) \leq r_i$. Then $B:= A \cap \bigcap_{j=3}^n B(x_j,r_j) \in \mathcal{\tilde{E}}_2(X)$ by Lemma~\ref{Lem:externally (n-k)-hyperconvex}. Moreover, for $i=1,2$, we have $A_i := A \cap B(x_i,r_i) \in \mathcal{\tilde{E}}_3(X)$ by Lemma~\ref{Lem:externally (n-k)-hyperconvex} and $A_i \cap B = A \cap B(x_i,r_i) \cap \bigcap_{j=3}^n B(x_j,r_j) \neq \emptyset$ by Lemma~\ref{Lem:almost externally (n+1)-hyperconvex}. Furthermore, observe that $A_1,A_2,B \subset A$ and $A$ is almost $n$-hyperconvex. Hence by Lemma~\ref{Lem:intersection of three almost externally hyperconvex sets} we get $$
A \cap \bigcap_{i=1}^n B(x_i,r_i)  = A_1 \cap A_2 \cap B \neq \emptyset,$$
as desired.
\end{proof}

This proves Theorem~\ref{Thm:AlmostAndComplete} and as a consequence we get the following result for the metric completion of an $n$-hyperconvex metric space.

\begin{Cor}
	Let $X$ be an $n$-hyperconvex metric space for $n \geq 3$. Then its metric completion is $n$-hyperconvex as well.
\end{Cor}

According to Proposition~\ref{Prop:almost n-hyperconvex} and the following Lemma~\ref{Lem:almost externally 2-hypercovnex}, we get that in a complete, 4-hyperconvex metric space $X$, there is no difference between $\mathcal{\tilde{E}}_n(X)$ and $\mathcal{E}_n(X)$ for $n \geq 2$.

\begin{Lem}\label{Lem:almost externally 2-hypercovnex}
	Let $X$ be a complete, 4-hyperconvex metric space. If $A \in \mathcal{\tilde{E}}_2(X)$ then we have $A \in \mathcal{E}_2(X)$.
\end{Lem}

\begin{proof}
	Let $\{ B(x_i,r_i) \} _{i=1}^2$ with $d(x_1,x_2) \leq r_1 + r_2$ and $d(x_i,A) \leq r_i$. Then we have $B(x_i,r_i) \in \mathcal{\tilde{E}}_3(X)$ by Lemma~\ref{Lem:externally (n-k)-hyperconvex} and $A \cap B(x_i,r_i) \neq \emptyset$ by Lemma~\ref{Lem:almost externally (n+1)-hyperconvex}. Therefore we get $$B(x_1,r_1) \cap B(x_2,r_2) \cap A \neq \emptyset$$ by Lemma~\ref{Lem:intersection of three almost externally hyperconvex sets}.
\end{proof}

We now turn our attention to complete, $4$-hyperconvex metric spaces and eventually prove that they are $n$-hyperconvex for every $n$.

\begin{Lem}\label{Lem:intersection of externally 3-hyperconvex is externally 2-hyperconvex}
	Let $X$ be a complete, 4-hyperconvex metric space. If $A_0 \in \mathcal{E}_3(X)$, $A_1 \in \mathcal{E}_2(X)$ and $A_0 \cap A_1 \neq \emptyset$, then  $A_0 \cap A_1 \in \mathcal{E}_2(X)$.
\end{Lem}

\begin{proof}
	Let $B(x_1,r_1),B(x_2,r_2)$ be closed balls such that  $d(x_1,x_2) \leq r_1 + r_2$ and $d(x_i,A_0 \cap A_1) \leq r_i$. Define $A_2:= B(x_1,r_1) \cap B(x_2,r_2) \in \mathcal{E}_2(X)$. Since for $k=0,1$, the sets $A_k$ are externally 2-hyperconvex, we have
	$$A_2 \cap A_k = B(x_1,r_1) \cap B(x_2,r_2) \cap A_k \neq \emptyset.$$
	Therefore, we get
	$$A_0 \cap A_1 \cap B(x_1,r_1) \cap B(x_2,r_2) = A_0 \cap A_1 \cap A_2 \neq \emptyset$$
	by Lemma~\ref{Lem:intersection of three almost externally hyperconvex sets}.
\end{proof}

\begin{Prop}\label{Prop:4-hyperconvex is n-hyperconvex}
	Let $X$ be a complete, 4-hyperconvex metric space. Then
\begin{enumerate}[(i)]
	\item $X$ is $n$-hyperconvex for every $n \in \mathbb{N}$,
	\item if $A \in \mathcal{E}_2(X)$, we have $A \in \mathcal{E}_n(X)$ for every $n \in \mathbb{N}$, and
	\item if $A_1, \ldots , A_m  \in \mathcal{E}_2(X)$ with $A_i \cap A_j \neq \emptyset$ for every $i,j$, then $\bigcap_{i=1}^m A_i \in \mathcal{E}_n(X)$ for every $n \in \mathbb{N}$.
\end{enumerate}
\end{Prop}

\begin{proof}
	In order to show (i), we will prove by induction that the following claim is true.
	\begin{Cl*}
		For $\{ B(x_i,r_i) \}_{i=1}^n$ with $d(x_i,x_j) \leq r_i + r_j$, we have $\bigcap_{i=1}^n B(x_i,r_i) \in \mathcal{E}_2(X)$.
	\end{Cl*}
	This clearly holds for $n=2$. For $n \geq 2$, consider $\{ B(x_i,r_i) \}_{i=1}^{n+1}$ with $d(x_i,x_j) \leq r_i + r_j$. Observe that $B(x_1,r_1), B(x_2,r_2) \in \mathcal{E}_3(X)$ and by the induction hypothesis $A = \bigcap_{i=3}^{n+1} B(x_i,r_i) \in \mathcal{E}_2(X)$ and $B(x_i,r_i) \cap A \neq \emptyset$ for $i=1,2$. Hence by Lemma~\ref{Lem:intersection of three almost externally hyperconvex sets} we get $$\bigcap_{i=1}^{n+1} B(x_i,r_i) = B(x_1,r_1) \cap B(x_2,r_2) \cap A \neq \emptyset.$$ 
	Again by the induction hypothesis we have $A' := \bigcap_{i=1}^n B(x_i,r_i) \in \mathcal{E}_2(X)$ and therefore we conclude that
	$$\bigcap_{i=1}^{n+1} B(x_i,r_i) = A' \cap B(x_{n+1},r_{n+1}) \in \mathcal{E}_2(X)$$
by	Lemma~\ref{Lem:intersection of externally 3-hyperconvex is externally 2-hyperconvex}. 

	For (ii), we also do induction on $n$. Let $\{ B(x_i,r_i) \}_{i=1}^{n+1}$ be a collection of balls with $d(x_i,x_j) \leq r_i + r_j$ and $d(x_i,A) \leq r_i$. We have $A_0 := \bigcap_{i=1}^n B(x_i,r_i)$, $A_1 := B(x_{n+1},r_{n+1}) \in \mathcal{E}_3(X)$ by (i) and $A_0 \cap A \neq \emptyset$ by the induction hypothesis. Hence we get $$\bigcap_{i=1}^{n+1} B(x_i,r_i) \cap A = A_0 \cap A_1 \cap A \neq \emptyset$$ by Lemma~\ref{Lem:intersection of three almost externally hyperconvex sets}.
	
	Finally, statement (iii) is a consequence of (ii), Lemma~\ref{Lem:intersection of three almost externally hyperconvex sets} and Lemma~\ref{Lem:intersection of externally 3-hyperconvex is externally 2-hyperconvex}.
\end{proof}

\begin{proof}[Proof of Theorem~\ref{Thm:FourToFinite}.]
	Statement (i) was shown in Proposition~\ref{Prop:4-hyperconvex is n-hyperconvex}(i).
	Therefore, we will start proving (ii). Let $A \in \mathcal{E}_4(X)$ and let $\{ B(x_i,r_i) \}_{i=1}^n$ be a collection of balls with $d(x_i,x_j) \leq r_i + r_j$ and $d(x_i,A) \leq r_i$. Define $A_i := B(x_i,r_i) \cap A \in \mathcal{E}_3(A)$. Moreover, we have $A_i \cap A_j = B(x_i,r_i) \cap B(x_j,r_j) \cap A \neq \emptyset$ and therefore we get $$\bigcap_{i=1}^n B(x_i,r_i) \cap A = \bigcap_{i=1}^n A_i \neq \emptyset$$ by Proposition~\ref{Prop:4-hyperconvex is n-hyperconvex}(iii). This is $A \in \mathcal{E}_n(X)$.
	
	Finally, let $A \in \mathcal{W}_4(X)$ and let $\{ B(x_i,r_i) \}_{i=1}^{n-1}$ be a collection of balls in $A$ with $d(x_i,x_j) \leq r_i + r_j$ and let $x \in X$, $r \geq 0$ with $d(x,A) \leq r$, $d(x,x_i) \leq r + r_i$. Then we have $A':= A \cap B(x,r) \in \mathcal{E}_3(A)$ and therefore $A' \in \mathcal{E}_{n-1}(A)$ by Proposition~\ref{Prop:4-hyperconvex is n-hyperconvex}(ii). Moreover, $B(x_i,r_i) \cap A' \neq \emptyset$ and therefore $$\bigcap_{i=1}^{n-1} B(x_i,r_i) \cap B(x,r) \cap A = \bigcap_{i=1}^{n-1} B(x_i,r_i) \cap A' \neq \emptyset.$$
	This proves $A \in \mathcal{W}_n(X)$ and hence establishes (iii).
\end{proof}

\begin{Expl}\label{Expl:externally n-hyperconvex but not hyperconvex}
	Consider $c_0 \subset l_\infty(\mathbb{N})$, the subspace of all null sequences. As it is mentioned in \cite{Lin}, $c_0$ is finitely hyperconvex but not hyperconvex. Moreover, $c_0$ is externally $n$-hyperconvex in $l_\infty(\mathbb{N})$ for every $n$. Indeed, we will show that $c_0 \in \mathcal{E}_2(l_\infty(\mathbb{N}))$ and then use Proposition~\ref{Prop:4-hyperconvex is n-hyperconvex}(ii). Let $\mathbf{x}:=(x_n)_{n \in \mathbb{N}}$, $\mathbf{y}:=(y_n)_{n \in \mathbb{N}} \in l_\infty(\mathbb{N})$ with $d_\infty(\mathbf{x},\mathbf{y}) \leq r + s$, $d_\infty(\mathbf{x},c_0) \leq r$  and $d_\infty(\mathbf{y},c_0) \leq s$. For every $n \in \mathbb{N}$, choose some
	$$z_n \in B(x_n,r) \cap B(y_n,s) \text{ with } |z_n|= \inf \{|\zeta| : \zeta \in B(x_n,r) \cap B(y_n,s)\}.$$
	Define $\mathbf{z}:=(z_n)_{n \in \mathbb{N}}$. Clearly, we have $d(x_n,z_n) \leq r$ and $d(y_n,z_n) \leq s$. Moreover, $\lim_{n \to \infty} z_n = 0$, since $\limsup_{n \to \infty} |x_n| \leq r$ and $\limsup_{n \to \infty} |y_n| \leq s$. Hence 
	$$\mathbf{z} \in B(\mathbf{x},r) \cap B(\mathbf{y},s) \cap c_0 \neq \emptyset.$$
	This shows that $c_0$ is externally $2$-hyperconvex in $l_\infty(\mathbb{N})$ and therefore it is externally $n$-hyperconvex in $l_\infty(\mathbb{N})$ for every $n$.
\end{Expl}


\section{Convexity} \label{Sec:Convexity}


The goal of this section is to prove $\sigma$-convexity for (weakly) externally hyperconvex subsets and therefore establish the implications $(i) \Rightarrow (ii)$ in Theorem~\ref{Thm:BicombingTheorem}. Afterwards, we conclude this section with the proof of Theorem~\ref{Thm:HellyWEH}.

\begin{Prop} \label{Prop:convexity of EH}
Suppose that $X$ is a metric space with a geodesic bicombing $\sigma$ such that the geodesics $\sigma_{xy}$ are straight curves. Moreover, let $E \in \mathcal{E}_2(X)$. Then $E$ is $\sigma$-convex.
\end{Prop}

\begin{proof}
Assume by contradiction that $E$ is not $\sigma$-convex. Then there are $x,y \in E$ such that 
\[
\sigma_{xy}([0,1]) \nsubseteq E.
\]
Hence there are $0 \leq t_1 < t_2 \leq 1$ such that $\sigma_{xy}([t_1,t_2]) \cap E = \{\bar{x},\bar{y} \}$ with $\bar{x} = \sigma_{xy}(t_1)$ and $\bar{y} = \sigma_{xy}(t_2)$.
Let
\[
R := \max_{t \in [t_1,t_2]} d(\sigma_{xy}(t),E) > 0.
\]
Define
\[
s := \min \{ t' \in [t_1,t_2] : d(\sigma_{xy}(t'),E) = R \} \text{ and } z:= \sigma_{xy}(s).
\]
In particular, we have $d(E, z) = R$. For an arbitrarily chosen $\eps \in (0, \frac{R}{2 d(x,y)})$,
we set $z^-:= \sigma_{xy}(s-\eps)$ and $z^+:= \sigma_{xy}(s+\eps)$. We then have $R^-:= d(z^-,E) < R$ and $R^+:= d(z^+,E) \le R$. Moreover, by the choice of $\eps$ we get 
$$R^-= d(z^-,E) \geq d(z,E) - d(z,z^-) = R - \eps d(x,y) > \tfrac{R}{2}$$
and similarly $R^+ > \frac{R}{2}$. In particular, we have 
$$d(z^-,z^+) = 2 \eps d(x,y) < \tfrac{R}{2} + \tfrac{R}{2} < R^- + R^+$$
and therefore by external 2-hyperconvexity of $E$ we can pick
\[
e \in E \cap B(z^-, R^-) \cap B(z^+, R^+).
\]
Now, since the curve $t \mapsto \sigma_{xy}(t)$ is straight in $X$, it follows that
\[
d(E, \sigma_{xy}(s))
\le d(e, \sigma_{xy}(s))  
\le \tfrac{1}{2}\left(d(e,z^-) + d(e,z^+) \right) 
\le \tfrac{1}{2}\left(R^- + R^+ \right)
< R
\]
which is a contradiction to the definition of $s$. This shows that $E$ is $\sigma$-convex.
\end{proof}

\begin{Prop} \label{Prop:convexity of WEH}
Suppose that $X$ is a metric space with unique straight curves and let $\sigma^X$ be a convex geodesic bicombing on $X$. Assume moreover that $W \in \mathcal{W}_3(X)$ and $W$ possesses a consistent geodesic bicombing $\sigma^W$. Then it follows that $\sigma^W = \sigma^X|_{W \times W \times [0,1]}$. In particular, $W$ is $\sigma^X$-convex.
\end{Prop}

\begin{proof}
We claim that the geodesics of $\sigma^W$ are straight curves in $X$.
Consider $z \in X$ and let $s:=d(z,W)$. For $w , w' \in W$, let $\bar{w} \in B(z,s) \cap W$ and $\bar{w}' \in B(z,s) \cap W$ be such that 
\begin{align*}
s + d(\bar{w},w) &= d(z,\bar{w}) + d(\bar{w},w)=d(z,w) \ \text{ and } \  \\
s + d(\bar{w}',w') &= d(z,\bar{w}') + d(\bar{w}',w')=d(z,w').
\end{align*}
We have 
\begin{equation}\label{Eq:convexity of WEH}
d(z,\sigma^W_{ww'}(t))
\le d(z,\sigma^W_{\bar{w}\bar{w}'}(t)) + d(\sigma^W_{\bar{w}\bar{w}'}(t),\sigma^W_{ww'}(t)). 
\end{equation}
To bound the first term on the right-hand side of \eqref{Eq:convexity of WEH}, note that $E := B(z,s) \cap W \in \mathcal{E}_2(W)$ and thus $E$ is in particular $\sigma^W$-convex by Proposition~\ref{Prop:convexity of EH}, that is $d(z,\sigma^W_{\bar{w}\bar{w}'}(t))=s$. On the other hand, to bound the second term on the right-hand side of \eqref{Eq:convexity of WEH}, note that by convexity of $\sigma^W$, we get 
\[
d(\sigma^W_{\bar{w}\bar{w}'}(t),\sigma^W_{ww'}(t))
\le (1-t) d(\bar{w},w) + t d(\bar{w}',w').
\]
Putting those two estimates together, we obtain that
\[
d(z,\sigma^W_{ww'}(t)) \le (1-t) d(z,w) + t d(z,w').
\]
It follows that the consistent geodesic bicombing $\sigma^W$ consists of straight curves in $X$. Since straight curves in $X$ are unique, $\sigma^W$ must coincide with the restriction  $\sigma^X|_{W \times W \times [0,1]}$. Therefore $W$ is $\sigma^X$-convex.
\end{proof}

This result was already announced in the proof of \cite[Proposition~4.8]{MieP}.

\begin{Rem} \label{Rem:LinearOnCells}
In \cite{Lan}, Lang proves that the injective hull of certain discretely geodesic metric spaces (including hyperbolic groups) has the structure of a locally finite polyhedral complex of finite combinatorial dimension. The cells are weakly externally hyperconvex subsets of this complex \cite[Remark~4.3]{MieP} and from Proposition \ref{Prop:convexity of WEH} it follows now that the unique consistent convex geodesic bicombing on this complex is linear inside the cells.
\end{Rem}

\begin{Lem}\label{Lem:DinstanceToConvexSet}
	Let $X$ be a metric space with a geodesic bicombing $\sigma$ and let $A \subset X$ be $\sigma$-convex. Then the following holds:
	\begin{enumerate}[(i)]
	\item For all $x,y \in X$ and $t \in [0,1]$ we have $d(\sigma_{xy}(t),A) \leq (1-t)d(x,A) + t d(y,A)$.
	\item If $\sigma$ is consistent, then for all $x,y \in X$ the function $t \mapsto d(\sigma_{xy}(t),A)$ is convex.
\end{enumerate}	
\end{Lem}

\begin{proof}
	We prove (ii). For $s_1,s_2 \in [0,1]$ and $\epsilon > 0$, there are $p,q \in A$ with $d(\sigma_{xy}(s_1),p) \leq d(\sigma_{xy}(s_1),A) + \epsilon$ and $d(\sigma_{xy}(s_2),q) \leq d(\sigma_{xy}(s_2),A) + \epsilon$. We get
\begin{align*}
	d(\sigma_{xy}((1-t)s_1 + t s_2) , A) &\leq d(\sigma_{\sigma_{xy}(s_1)\sigma_{xy}(s_2)}(t),\sigma_{pq}(t)) \\
	&\leq (1-t) d(\sigma_{xy}(s_1),A) + t d(\sigma_{xy}(s_2),A) + \epsilon
\end{align*}	
	for all $\epsilon>0$. Finally, let $\epsilon \downarrow 0$.
\end{proof}

\begin{Cor}\label{Cor:DistanceToWEHSet}
	Suppose that $X$ is a proper metric space of finite combinatorial dimension which admits a consistent geodesic bicombing $\sigma$. Then for all $x,y \in X$ and $W \in \mathcal{W}_3(X)$ the function $t \mapsto d(\sigma_{xy}(t),W)$ is convex.
\end{Cor}

\begin{proof}[Proof of Theorem~\ref{Thm:HellyWEH}]
By Proposition~\ref{Prop:convexity of WEH}, weakly externally hyperconvex subsets of $l_{\infty}^n$ are convex in the usual sense. By the classical Helly theorem, it follows that $\mathcal{W}(l_{\infty}^n)$ is a Helly family of order $n+1$. It thus remains to be proved that $n+1$ is optimal. In other words, we need to find $\mathcal{A} := \{A_1,\dots,A_{n+1}\}\subset \mathcal{W}(l_{\infty}^n)$ such that 
\begin{enumerate}[(a)]
\item one has $\bigcap_{i=1}^{n+1} A_i = \emptyset$ and 
\item for any $j \in \{1,\dots,n+1\}$, one has $ B_j := \bigcap_{A \in \mathcal{A} \setminus \{A_j\}} A \neq \emptyset$.
\end{enumerate}
We can define the following half-spaces
\begin{align*}
A_1 &:= \{ x \in \R^n : x_n \ge 0\}, \\
A_2 &:= \{ x \in \R^n : x_{n-1} - x_n \ge 0\}, \\
A_3 &:= \{ x \in \R^n : x_{n-2} - x_{n-1} \ge 0\}, \\
& \ \ \vdots \\
A_{n-1} &:= \{ x \in \R^n : x_2 - x_3 \ge 0\}, \\
A_n &:= \{ x \in \R^n : x_1 - x_2 \ge 0\}, \\
A_{n+1} &:= \{ x \in \R^n : x_1 + x_2 \le -1 \}.
\end{align*}
Note if $x \in \bigcap_{i=1}^{n} A_i$, then $x_1,\dots,x_n \ge 0$ and thus $\bigcap_{i=1}^{n+1} A_i = \emptyset$ as desired. Moreover, note that
\begin{align*}
(\phantom{-}0,-5,-5, \dots,-5,-5,-5)  &\in B_1, \\
(\phantom{-}0,-5,-5, \dots,-5,-5,\phantom{-}0)  &\in B_2, \\
(-5,-5,-5,\dots,-5,\phantom{-}0,\phantom{-}0)  &\in B_3, \\
& \ \  \vdots \\
(-5,-5,\phantom{-}0, \dots,\phantom{-}0,\phantom{-}0,\phantom{-}0)  &\in B_{n-1}, \\
(-5,\phantom{-}0,\phantom{-}0, \dots,\phantom{-}0,\phantom{-}0,\phantom{-}0)  &\in B_n, \\
(\phantom{-}0,\phantom{-}0,\phantom{-}0,\dots,\phantom{-}0,\phantom{-}0,\phantom{-}0)  &\in B_{n+1}.
\end{align*}
This concludes the proof.
\end{proof}


\section{Local to global} \label{Local to global}


In \cite{Mie2}, the first author showed that a uniformly locally hyperconvex metric space with a geodesic bicombing is hyperconvex. We will now extend this result to the classes of weakly externally hyperconvex and externally hyperconvex subsets. To this end, we will first list some results that we need afterwards.

\begin{Prop}\label{Prop:IntersectionExternallyHyperconvex}\cite[Proposition~1.2]{Mie}
	Let $(X,d)$ be a hyperconvex space and $\{A_i\}_{i \in I}$ a family of pairwise intersecting externally hyperconvex subsets such that one of them is bounded. Then $\bigcap_{i \in I} A_i \neq \emptyset$.
\end{Prop}

\begin{Lem}\label{Lem:IntersectionWEH}\cite[Proposition~2.11]{MieP}
	Let $X$ be a metric space and $A \in \mathcal{E}(X), Y \in \mathcal{W}(X)$ such that $A\cap Y \neq \emptyset$. Then we have $A\cap Y \in \mathcal{W}(X)$.
\end{Lem}

\begin{Lem}\label{Lem:WEHInNbhd} \cite[Lemma~2.18]{MieP}
	Let $X$ be a hyperconvex metric space and $A \subset X$. Assume that there is some $s > 0$ such that $A \in \mathcal{W}(B(A,s))$. Then we have $A \in \mathcal{W}(X)$.
\end{Lem}

\begin{Lem}\label{Lem:ExternallyHyperconvexInNbhd} \cite[Lemma~2.19]{MieP}
	Let $X$ be a hyperconvex metric space and $A \subset X$. Assume that there is some $s > 0$ such that $A \in \mathcal{E}(B(A,s))$. Then we have $A \in \mathcal{E}(X)$.
\end{Lem}

	Let $X$ be a metric space. A subset $A \subset X$ is called \emph{locally (weakly) externally hyperconvex} in $X$ if for all $x \in A$ there is some $r_x > 0$ such that $A \cap B(x,r_x)$ is (weakly) externally hyperconvex in $B(x,r_x)$.
	If we can choose $r_x=r >0$ for all $x \in A$, we call $A$ \emph{uniformly locally (weakly) externally hyperconvex} in $X$.

\begin{Lem}\label{Lem:LocallyExternallyHyperconvex}
	Let $X$ be a hyperconvex metric space. Then $A \subset X$ is locally (weakly) externally hyperconvex in $X$ if and only if for every $x \in A$ there is some $r_x > 0$ such that  $A \cap B(x,r_x)$ is (weakly) externally hyperconvex in $X$.
\end{Lem}

\begin{proof}
	If $A \cap B(x,r_x)$ is (weakly) externally hyperconvex in $B(x,r_x)$ then $$A \cap B(x,\tfrac{r_x}{2}) = (A \cap B(x,r_x)) \cap B(x,\tfrac{r_x}{2})$$ is (weakly) externally hyperconvex in $$B(A \cap B(x,\tfrac{r_x}{2}),\tfrac{r_x}{2}) \subset B(x,r_x)$$ and therefore $B(x,\tfrac{r_x}{2})$ is (weakly) externally hyperconvex in $X$ by Lemma~\ref{Lem:WEHInNbhd} and Lemma~\ref{Lem:ExternallyHyperconvexInNbhd}. The converse is obvious.
\end{proof}

\begin{Lem}\label{Lem:HyperconvexBalls} \cite[Lemma~3.2]{Mie2}
	Let $X$ be a metric space with the property that every closed ball $B(x,r)$ is hyperconvex, then $X$ is itself hyperconvex.
\end{Lem}

\begin{Lem}\label{Lem:BoundedExternallyHyperconvex}
	Let $X$ be a metric space and $A \subset X$. Then $A \in \mathcal{E}(X)$ if and only if $A \cap B(z,R) \in \mathcal{E}(B(z,R))$ for all $z \in X , R \geq d(z,A)$.
\end{Lem}

\begin{proof}
If $A \in \mathcal{E}(X)$ then $A \cap B(z,R) \in \mathcal{E}(B(z,R))$ by Lemma~\ref{Lem:IntersectionWEH}.

For the other direction, first observe that $A$ is hyperconvex by Lemma~\ref{Lem:HyperconvexBalls}. Let $x_i \in X$ with $d(x_i,A) \leq r_i$ and $d(x_i,x_j) \leq r_i + r_j$. Define $B_i := A \cap B(x_i,r_i)$. Since for fixed $i,j$ and some $s>0$ there are $z \in X$ and $R>0$ such that $B(x_i,r_i+s), B(x_j,r_j+s) \subset B(z,R)$, we have $B_i \cap B_j \neq \emptyset$ and $B_i \in \mathcal{E}(A)$ by Lemma~\ref{Lem:ExternallyHyperconvexInNbhd}. Hence we can conclude
	$$A \cap \bigcap_i B(x_i,r_i) = \bigcap_i B_i \neq \emptyset$$
by Proposition~\ref{Prop:IntersectionExternallyHyperconvex}.
\end{proof}

\begin{Lem}\label{Lem:BoundedWEH}
	Let $X$ be a metric space and $A \subset X$. Then $A \in \mathcal{W}(X)$ if and only if $A \cap B(z,R) \in \mathcal{W}(B(z,R))$ for all $z \in X , R \geq d(z,A)$.
\end{Lem}

\begin{proof}
If $A \in \mathcal{W}(X)$ then $A \cap B(z,R) \in \mathcal{W}(B(z,R))$ by Lemma~\ref{Lem:IntersectionWEH}.

For the other direction, first observe that $A$ is hyperconvex by Lemma~\ref{Lem:HyperconvexBalls}. Let $x \in X$, $x_i \in A$ with $d(x,A) \leq r$, $d(x,x_i) \leq r + r_i$ and $d(x_i,x_j) \leq r_i + r_j$. Define $\bar{B} := A \cap B(x,r)$ and $B_i := A \cap B(x_i,r_i)$. Since for fixed $i,j$ and some $s>0$ there are $z \in X$ and $R>0$ such that $B(x,r+s), B(x_i,r_i+s), B(x_j,r_j+s) \subset B(z,R)$, we have $\bar{B} \cap B_i \cap B_j \neq \emptyset$ and $\bar{B}, B_i \in \mathcal{E}(A)$ by Lemma~\ref{Lem:ExternallyHyperconvexInNbhd}. Hence we can conclude
	$$A \cap B(x,r) \cap \bigcap_i B(x_i,r_i) = \bar{B} \cap \bigcap_i B_i \neq \emptyset$$
by Proposition~\ref{Prop:IntersectionExternallyHyperconvex}.
\end{proof}

\begin{Lem}\label{Lem:ProximinalInNbhd}
	Let $X$ be a hyperconvex metric space and $A\cap B(z,R) \in \mathcal{W}(B(z,R))$ for all $z \in B(A,R)$. Then $A$ is proximinal in $B\left(A, \frac{5}{4}R\right)$, especially for all $\bar{z} \in B\left(A, \frac{5}{4}R\right)$ we have $A \cap B\left(\bar{z},\frac{5}{4}R\right) \neq \emptyset$.
\end{Lem}

\begin{proof}
	Let $\bar{z} \in B(A, \frac{5}{4}R)$ and $d(\bar{z},A) = R + t$ with $t \in \left[0,\frac{R}{4}\right]$. Then for every $\epsilon \in \left(0, \frac{R}{4}\right]$ there is some $z_\epsilon \in B(A,R)$ with $d(z,z_\epsilon) \leq t + \epsilon$. Especially, we have $d(z_\epsilon,z_{\frac{R}{4}}) \leq R$, i.e. $z_\epsilon \in B(z_{\frac{R}{4}},R)$, and by our assumption there is some proximinal non-expansive retraction $\rho \colon B(z_{\frac{R}{4}},R) \to A \cap B(z_{\frac{R}{4}},R)$.
	Since $d(\rho(z_\epsilon),\bar{z}) \leq R + t + \epsilon$ and $d(\rho(z_\epsilon),\rho(z_{\epsilon'})) \leq d(z_\epsilon,z_{\epsilon'}) \leq 2t + \epsilon + \epsilon'$ there is some
$$y \in B(\bar{z},R) \cap \bigcap_{\epsilon} B\left(\rho(z_\epsilon),t + \epsilon\right).$$
Finally, we have $d(y,A) = t \leq R$ and hence there is some $\bar{y} \in A$ with $d(\bar{y},y)=t$ and thus $d(\bar{y},\bar{z}) \leq d(\bar{y},y) + d(y,\bar{z}) \leq t + R = d(\bar{z},A)$.
\end{proof}

We are now able to prove our local-to-global result for (weakly) externally hyperconvex subsets of a metric space with a geodesic bicombing.

\begin{Prop}\label{Prop:LocallyExternallyHyperconvexSigmaConvex}
	Let $X$ be a hyperconvex metric space with a geodesic bicombing $\sigma$. If $A \subset X$ is $\sigma$-convex and uniformly locally (weakly) externally hyperconvex, then $A$ is (weakly) externally hyperconvex in $X$.
\end{Prop}

\begin{proof}
The two cases are similar but we prove them separately to be precise.
\\ \mbox{} \\
\noindent \textbf{\textit{Case 1.}} Let $A$ be $\sigma$-convex and uniformly locally externally hyperconvex. 

	Let $\mathsf{P}(R)$ be the property given by the following expression:
\begin{tabbing}
$\mathsf{P}(R)$: \= $  \forall z \in B(A,R) : A \cap B(z,R) \in \mathcal{E}(B(z,R)).$
\end{tabbing}

	We want to show that this is true for every $R>0$. Then, by Lemma~\ref{Lem:BoundedExternallyHyperconvex}, $A$ is externally hyperconvex in $X$.
	
	Since $A$ is uniformly locally externally hyperconvex $\mathsf{P}(R)$ clearly holds for some small $R > 0$. Hence we need to show that $\mathsf{P}(R) \Rightarrow \mathsf{P}(\frac{5}{4}R)$.
	
	For $z \in B\left(A, \frac{5}{4}R\right)$, let
\begin{align*}
	X':=B\left(z,\tfrac{5}{4}R\right) \text{, } A' := A \cap X' \text{ and } B'(x,r) := B(x,r) \cap X' \text{ for } x\in X'.
\end{align*}
	By Lemma~\ref{Lem:ExternallyHyperconvexInNbhd} it is enough to show that
\begin{align}\label{Eq:ExternallyHyperconvex}
	A' \in \mathcal{E}\left(B'\left(A',\tfrac{R}{4}\right)\right).
\end{align}
	First note that $A' \neq \emptyset$ by Lemma~\ref{Lem:ProximinalInNbhd}.
	
	We establish now the following property for all $R'>0$:
\begin{tabbing}
$\mathsf{P}'(R')$: \= $\forall \{x_i\}_i \subset B'(A',\tfrac{R}{4})$ with $d(x_i,x_j) \leq r_i + r_j$, $d(x_i,A') \leq r_i$ and $r_i \leq R'$, \\
\> we have $A' \cap \bigcap_i B'(x_i,r_i) \neq \emptyset.$
\end{tabbing}

	Observe that $\mathsf{P}'(\frac{5}{2}R)$ implies (\ref{Eq:ExternallyHyperconvex}) since balls with center in $B(z,\frac{5}{4}R)$ and radius bigger than $\frac{5}{2}R$ contain $B(z,\frac{5}{4}R)$.
\\ \mbox{} \\
\noindent \textbf{Step I.} $\mathsf{P}'(\frac{R}{2})$ holds.
\\ \mbox{} \\
In this case we have $d(x_i,z)\leq \frac{5}{4}R$, $d(x_i,x_j) \leq R$ and hence there is some $\bar{z}\in B(z,R) \cap \bigcap_i B(x_i,\frac{R}{2})$. Then we have $d(\bar{z},A) \leq d(\bar{z},x_i) + d(x_i,A) \leq R$ and $A \cap B(\bar{z},R) \in \mathcal{E}(B(\bar{z},R))$ by $\mathsf{P}(R)$.

Clearly, $z, x_i \in B(\bar{z},R)$. Therefore, if we have $d(x_i,A \cap B(\bar{z},R)) \leq r_i$ and $d(z ,A \cap B(\bar{z},R)) \leq \frac{5}{4}R$, there is some
$$y \in A \cap B(\bar{z},R) \cap B\left(z,\tfrac{5}{4}R\right) \cap \bigcap_i B(x_i,r_i) \neq \emptyset$$
and thus $y \in A' \cap \bigcap_i B'(x_i,r_i) \neq \emptyset$.

Indeed, since for every $\epsilon \in \left(0,\frac{R}{4}\right]$ there is some $a_\epsilon \in A'= A \cap B\left(z,\frac{5}{4}R\right)$ with $d(a_\epsilon,x_i) \leq \min \left\{r_i,\frac{R}{4}\right\}+\epsilon$. It follows
\begin{align*}
	d(a_\epsilon,\bar{z}) \leq d(a_\epsilon,x_i) + d(x_i,\bar{z}) \leq \tfrac{R}{4} + \epsilon + \tfrac{R}{2} \leq R.
\end{align*}
and therefore $a_\epsilon \in A \cap B(\bar{z},R)$. Hence we get
\begin{align*}
	d(x_i,A \cap B(\bar{z},R)) &\leq d(x_i,a_\epsilon) \leq r_i+\epsilon \text{, for all } \epsilon>0 \text{, i.e.} \\
	d(x_i,A \cap B(\bar{z},R)) &\leq r_i \text{, and} \\
	d(z ,A \cap B(\bar{z},R)) &\leq d(z,a_\epsilon) \leq \tfrac{5}{4}R,
\end{align*}
as desired.
\\ \mbox{} \\
\noindent \textbf{Step II.} We establish $\mathsf{P}'(R') \Rightarrow \mathsf{P}'(2R')$.
\\ \mbox{} \\
	Let $y_{ij}=\sigma_{x_ix_j}\left( \frac{1}{2} \right) $ and $z_i = \sigma_{x_iz}\left( \frac{1}{2}\right) $.
	We get
\begin{align*}
	d(z_i,A) &\leq d(z_i,x_i) + d(x_i,A) \leq \tfrac{5}{8}R + \tfrac{R}{4} \leq R, \\
	d(y_{ij},z_i) &\leq \tfrac{1}{2}d(x_j,z) \leq \tfrac{5}{8}R \leq R, \\
	d(z,z_i) &= \tfrac{1}{2}d(z,x_i) \leq \tfrac{5}{8}R \leq R,
\end{align*}
	i.e. $y_{ij},z \in B(z_i,R)$, and $A \cap B(z_i,R) \in \mathcal{E}(B(z_i,R))$ by $\mathsf{P}(R)$.
	Note that there is some $a \in A' = A \cap B(z,\frac{5}{4}R)$ with $d(a,x_i) \leq d(x_i,A') + \frac{R}{8} \leq \frac{3}{8}R$ and thus $$d(z_i,a) \leq d(z_i,x_i)+ d(x_i,a) \leq \tfrac{5}{8}R + \tfrac{3}{8}R = R.$$ Hence we get $a \in A_i := A \cap B(z_i,R) \cap B(z, \tfrac{5}{4}R) \neq \emptyset$ and therefore $A_i \in \mathcal{E}(B(z_i,R))$.
	Moreover, we have
\begin{align*}
	d(y_{ij},y_{ik}) &\leq \tfrac{1}{2}d(x_j,x_k) \leq \tfrac{r_j}{2} + \tfrac{r_k}{2}, \\
	d(y_{ij},A') &\leq  \tfrac{1}{2} \left( d(x_i,A') + d(x_j,A') \right) \leq \min\{ \tfrac{r_i}{2} + \tfrac{r_j}{2}, \tfrac{R}{4}\}
\end{align*}
	by Lemma~\ref{Lem:DinstanceToConvexSet}.
	For $\epsilon \in \left(0,\frac{R}{8} \right]$, let $a_\epsilon \in A'$ with $d(y_{ij},a_\epsilon) \leq d(y_{ij},A') + \epsilon$. Then we have $d(a_\epsilon,z_i) \leq d(a_\epsilon,y_{ij}) + d(y_{ij},z_i) \leq \frac{R}{4}+ \epsilon + \frac{5}{8}R \leq R$, i.e. $a_\epsilon \in A_i$, and therefore $$d(y_{ij},A_i) = d(y_{ij},A') \leq \min\{ \tfrac{r_i}{2} + \tfrac{r_j}{2}, \tfrac{R}{4}\}.$$
	Hence there are $\bar{x}_i \in B(A_i,\min\{\tfrac{r_i}{2},\tfrac{R}{4}\}) \cap B(z,\tfrac{5}{4}R) \cap \bigcap_j B(y_{ij},\frac{r_j}{2})$ with
\begin{align*}
	d(\bar{x}_i,\bar{x}_j) &\leq d(\bar{x}_i,y_{ij})+d(y_{ij},\bar{x}_j) \leq \tfrac{r_i}{2} + \tfrac{r_j}{2}, \\
	d(\bar{x}_i,A') &\leq d(\bar{x}_i,A_i) \leq \min\{\tfrac{r_i}{2},\tfrac{R}{4}\}.
\end{align*}
	Therefore by $\mathsf{P'}(R')$ we get $y\in A' \cap \bigcap_i B(\bar{x}_i,\frac{r_i}{2}) \subset A' \cap \bigcap_i B(x_i,r_i) \neq \emptyset$ as desired.
\\ \mbox{} \\
\noindent \textbf{\textit{Case 2.}} Let $A$ be $\sigma$-convex and uniformly locally weakly externally hyperconvex. 

Let $\mathsf{P}(R)$ be the property given by the following expression:
\begin{tabbing}
$\mathsf{P}(R)$: \= $\forall z \in B(A,R) : A \cap B(z,R) \in \mathcal{W}(B(z,R)).$
\end{tabbing}
	
	We want to show that this is true for every $R>0$. Then, by Lemma~\ref{Lem:BoundedWEH}, $A$ is weakly externally hyperconvex in $X$.

	Since $A$ is uniformly locally externally hyperconvex $\mathsf{P}(R)$ clearly holds for some small $R > 0$. Hence we need to show that $\mathsf{P}(R) \Rightarrow \mathsf{P}(\frac{5}{4}R)$.
	
	For $z \in B\left(A, \frac{5}{4}R\right)$, let
\begin{align*}
	X':=B\left(z,\tfrac{5}{4}R\right) \text{, } A' := A \cap X' \text{ and } B'(x,r) := B(x,r) \cap X' \text{ for } x\in X'.
\end{align*}
	By Lemma~\ref{Lem:WEHInNbhd} it is enough to show that 
\begin{align}\label{Eq:WEHInNbhd}
	A' \in \mathcal{W}\left( B'\left( A',\tfrac{R}{4}\right) \right).
\end{align}	
	First note that $A' \neq \emptyset$ by Lemma~\ref{Lem:ProximinalInNbhd}.
	
	We establish now the following property for all $R'>0$:
\begin{tabbing}
$\mathsf{P}'(R')$: \= $\forall x \in B'\left( A',\tfrac{R}{4}\right)$, $\{x_i \}_i \subset A'$ with $d(x,A') \leq r$, $d(x,x_i) \leq r + r_i$, \\
\> $d(x_i,x_j) \leq r_i + r_j$ and $r,r_i \leq R'$, \\
	\> we have $A' \cap B'(x,r) \cap \bigcap_i B'(x_i,r_i) \neq \emptyset.$
\end{tabbing}

	Observe that $\mathsf{P}'(\tfrac{5}{2}R)$ implies (\ref{Eq:WEHInNbhd}) since balls with center in $B(z,\tfrac{5}{4}R)$ and radius bigger than $\tfrac{5}{2}R$ contain $B(z,\tfrac{5}{4}R)$.
\\ \mbox{} \\
\noindent \textbf{Step I.} $\mathsf{P}'(\frac{R}{2})$ holds.
\\ \mbox{} \\
In this case we have $d(x,z),d(x_i,z)\leq \frac{5}{4}R$, $d(x,x_i),d(x_i,x_j) \leq R$ and hence there is some $\bar{z}\in B(z,R) \cap B(x,\frac{R}{2}) \cap \bigcap_i B(x_i,\frac{R}{2})$.

For $\epsilon \in (0,\frac{R}{4}]$, let $a_\epsilon \in A'$ with $d(x,a_\epsilon) \leq d(x,A') + \epsilon \leq \frac{R}{2}$. Then we have $d(\bar{z},A) \leq d(\bar{z},x) + d(x,a_\epsilon) \leq R$, $z, x, x_i \in B(\bar{z},R)$ and $a_\epsilon\in A \cap B(\bar{z},R) \cap B(z, \frac{5}{4}R) \neq \emptyset$. Hence $A \cap B(\bar{z},R) \cap B(z,\tfrac{5}{4}R) \in \mathcal{W}(B(\bar{z},R))$ by $\mathsf{P}(R)$ and Lemma~\ref{Lem:IntersectionWEH}, and $d(x,A \cap B(\bar{z},R) \cap B(z,\tfrac{5}{4}R)) \leq r$. Therefore there is some 
$$y \in A \cap B(\bar{z},R) \cap B(z,\tfrac{5}{4}R) \cap B(x,r) \cap \bigcap_i B(x_i,r_i) \neq \emptyset$$
and thus $y \in A' \cap B'(x,r) \cap \bigcap_i B'(x_i,r_i) \neq \emptyset$.
\\ \mbox{} \\
\noindent \textbf{Step II.} We establish $\mathsf{P}'(R') \Rightarrow \mathsf{P}'(2R')$.
\\ \mbox{} \\
	Let $y_i=\sigma_{xx_i}(\frac{1}{2})$, $y_{ij}=\sigma_{x_ix_j}(\frac{1}{2})$, $\bar{z} = \sigma_{xz}(\frac{1}{2})$ and $z_i = \sigma_{x_iz}(\frac{1}{2})$. Moreover, let $\bar{a}_\epsilon \in  A'$ such that $d(x,\bar{a}_\epsilon) \leq d(x,A') + \epsilon \leq \frac{3}{8}R$.
	We get
\begin{align*}
	d(\bar{z},\bar{a}_\epsilon) &\leq d(\bar{z},x) + d(x,\bar{a}_\epsilon) \leq \tfrac{5}{8}R + \tfrac{R}{4} \leq R, \\
	d(\bar{z},A) &\leq d(\bar{z},\bar{a}) \leq R, \\
	d(z_i,A) &\leq d(z_i,x_i) \leq \tfrac{5}{8}R \leq R, \\
	d(y_{ij},z_i) &\leq \tfrac{1}{2}d(x_j,z) \leq \tfrac{5}{8}R \leq R, \\
	d(y_i,z_i) &\leq \tfrac{1}{2}d(x,z) \leq \tfrac{5}{8}R \leq R, \\	
	d(z,z_i) &= \tfrac{1}{2}d(z,x_i) \leq \tfrac{5}{8}R \leq R,
\end{align*}
	i.e. $y_i, y_{ij}, z \in B(z_i,R)$, and $A \cap B(z_i,R) \in \mathcal{W}(B(z_i,R))$ by $\mathsf{P}(R)$.
	Define $A_i := A \cap B(z_i,R) \cap B(z, \tfrac{5}{4}R)$ and $\bar{A} := A \cap B(\bar{z},R) \cap B(z, \tfrac{5}{4}R)$. We have $y_{ij} \in A_i \neq \emptyset$ and $\bar{a}_\epsilon \in \bar{A} \neq \emptyset$ and therefore $A_i \in \mathcal{W}(B(z_i,R))$ and $\bar{A} \in \mathcal{W}(B(\bar{z},R))$ by Lemma~\ref{Lem:IntersectionWEH}. Furthermore, we get
\begin{align*}
	d(x,\bar{A}) &\leq d(x,\bar{a}_\epsilon) \leq r + \epsilon, \text{ i.e. } d(x,\bar{A}) \leq r, \\
	d(y_i,\bar{A}) &\leq  \tfrac{1}{2} d(x,\bar{A}) \leq \tfrac{r}{2} \text{ (Lemma~\ref{Lem:DinstanceToConvexSet})}, \\
	d(x,y_i) &= \tfrac{1}{2}d(x,x_i) \leq \tfrac{r}{2} + \tfrac{r_i}{2}, \\
	d(y_{i},y_{j}) &\leq \tfrac{1}{2}d(x_i,x_j) \leq \tfrac{r_i}{2} + \tfrac{r_j}{2}, \\
	d(y_{i},y_{ij}) &\leq \tfrac{1}{2}d(x,x_j) \leq \tfrac{r}{2} + \tfrac{r_j}{2}, \\	
	d(y_{ij},y_{ik}) &\leq \tfrac{1}{2}d(x_j,x_k) \leq \tfrac{r_j}{2} + \tfrac{r_k}{2}.
\end{align*}
	Hence there are $\bar{x} \in B(\bar{A},\frac{r}{2}) \cap B(x,\frac{r}{2}) \cap \bigcap_i B(y_i,\frac{r_i}{2})$, $\bar{x}_i \in A_i \cap B(y_i,\frac{r}{2}) \cap \bigcap_j B(y_{ij},\frac{r_j}{2})$ with 
\begin{align*}
	d(\bar{x},\bar{x}_i) &\leq d(\bar{x},y_i)+d(y_i,\bar{x}_i) \leq \tfrac{r}{2} + \tfrac{r_i}{2}, \\
	d(\bar{x}_i,\bar{x}_j) &\leq d(\bar{x}_i,y_{ij})+d(y_{ij},\bar{x}_j) \leq \tfrac{r_i}{2} + \tfrac{r_j}{2}.
\end{align*}	
	Therefore by $\mathsf{P'}(R')$ we get
	$$y\in A' \cap B(\bar{x},\tfrac{r}{2}) \cap \bigcap_i B(\bar{x}_i,\tfrac{r_i}{2}) \subset A' \cap B'(x,r) \cap \bigcap_i B'(x_i,r_i) \neq \emptyset$$
	as desired.
\end{proof}

\begin{proof}[Proof of Theorem~\ref{Thm:BicombingTheorem}]
	The first equivalence follows directly from Proposition~\ref{Prop:convexity of EH} and Proposition~\ref{Prop:LocallyExternallyHyperconvexSigmaConvex}.
	
	For the second equivalence note first that by the uniqueness of straight lines, $\sigma$ is a consistent geodesic bicombing and thus, if $A$ is a $\sigma$-convex subset of $X$, it possesses a consistent geodesic bicombing. The implications then follow from Proposition~\ref{Prop:convexity of WEH} and Proposition~\ref{Prop:LocallyExternallyHyperconvexSigmaConvex}.
\end{proof}

Together with Theorems~1.1 and 1.2 in \cite{DesL}, we conclude the following.

\begin{Cor}\label{Cor:LocallyWEHSigmaConvex}
	Let $X$ be a proper, hyperconvex metric space with finite combinatorial dimension and let $\sigma$ be the unique consistent convex geodesic bicombing on $X$. Then $A \subset X$ is weakly externally hyperconvex in $X$ if and only if $A$ is $\sigma$-convex and locally weakly externally hyperconvex in $X$.
\end{Cor}


\section{Retracts} \label{Sec:Retracts}


Weakly externally hyperconvex subsets were recognized as the proximinal 1-Lipschitz retracts by Esp\'inola in \cite{Esp}.

\begin{Prop}\cite[Theorem 3.6]{Esp}
Let $X$ be a hyperconvex metric space and let $A \subset X$ be non-empty. Then $A$ is a weakly externally hyperconvex subset of $X$ if and only if $A$ is a proximinal 1-Lipschitz retract of $X$.
\end{Prop}

Similarly, we can characterize externally hyperconvex subsets as 1-Lipschitz retracts with some further properties.

\begin{Prop}
Let $X$ be a hyperconvex metric space and $A \subset X$ a subset. Then $A$ is externally hyperconvex in $X$ if and only if there is a proximinal 1-Lipschitz retraction $\rho \colon X \to A$ with $d(\rho(x),y) \leq \max\{d(x,y),d(A,y)\}$ for all $x,y \in X$.
\end{Prop}

\begin{proof}
	First assume that there is a proximinal 1-Lipschitz retraction $\rho \colon X \to A$, with $d(\rho(x),y) \leq \max\{d(x,y),d(A,y)\}$ for all $x,y \in X$.
	
	Let $B(x_i,r_i)$ be a family of closed balls in $X$ with $d(x_i,x_j) \leq r_i+r_j$ and $d(x_i,A) \leq r_i$. Then by hyperconvexity of $X$ there is some $z \in \bigcap_i B(x_i,r_i).$
We get $d(\rho(z),x_i) \leq \max\{d(z,x_i),d(A,x_i)\} \leq r_i$ and hence
\begin{align*}
	\rho(z) \in A \cap \bigcap_i B(x_i,r_i) \neq \emptyset.
\end{align*}

	For the converse consider the set
\begin{align*}
	\mathcal{F} = \left\{ (Y, \rho) : Y \subset X, \rho \colon Y \to A \text{ a good retraction}  \right\},
\end{align*}
	i.e. $\rho \colon Y \to A$ is a retraction with $d(\rho(x),y) \leq \max\{d(x,y),d(A,y)\} \text{ for all } x,y \in X$. We endow $\mathcal{F}$ with the usual order relation $(Y,\rho) \preccurlyeq (Y',\rho')$ if and only if $Y \subset Y'$ and $\rho'|_Y = \rho$.
	
	Clearly, $\mathcal{F}$ is non-empty since $(A,\id) \in \mathcal{F}$ and for any chain $(Y_i,\rho_i)$ the element $(Y,\rho)$ with $Y = \bigcup Y_i$ and $\rho(x) = \rho_i(x)$ for $x \in Y_i$ is an upper bound. Hence there is some maximal element $(\bar{Y},\bar{\rho}) \in \mathcal{F}$.
	
	Assume that there is some $x_0 \in X \setminus \bar{Y}$. For all $x \in X$, define $r_x = \max\{d(x_0,x),d(A,x)\}$ and for all $y \in \bar{Y}$, define $s_y = d(x_0,y)$. We have
\begin{align*}
	d(x,x') &\leq d(x_0,x) + d(x_0,x') \leq r_x + r_{x'}, \\
	d(x,A) &\leq r_x, \\
	d(\rho(y),\rho(y')) &\leq d(y,y') \leq s_y + s_{y'}, \\
	d(\rho(y),x) &\leq \max \{ d(y,x) , d(A,x) \} \leq s_y + r_x.
\end{align*}
	Hence since $A$ is externally hyperconvex, there is some
\begin{align*}
	z \in A \cap \bigcap_{x \in X} B(x,r_x) \cap \bigcap_{y \in \bar{Y}} B(\rho(y),s_y).
\end{align*}
	But then $(\bar{Y}\cup \{x_0\}, \rho')$ with $\rho'(y) = \bar{\rho}(y)$ for $y \in \bar{Y}$ and $\rho'(x_0) = z$ is a strictly bigger element in $\mathcal{F}$, contradicting maximality of $(\bar{Y},\bar{\rho})$. 
		
	Therefore $\bar{Y}=X$ and $\bar{\rho} \colon X \to A$ is the desired retraction.
\end{proof}


From Sections~\ref{Sec:Convexity} to \ref{Sec:Retracts} and results of \cite{Esp,MieP} we get the following characterizations of (weakly) externally hyperconvex subsets of hyperconvex metric spaces.

\begin{Thm}\label{Thm:ExternallyHyperconvex}
Let $X$ be a hyperconvex metric space and $A \subset X$ a subset. Then the following are equivalent:
\begin{enumerate}[(i)]
\item $A$ is externally hyperconvex in $X$.
\item For all $x \in X$ and for all $r \geq d(x,A)$, $A \cap B(x,r)$ is externally hyperconvex in $B(x,r)$.
\item There is some $x \in X$ such that for all $r \geq d(x,A)$, $A \cap B(x,r)$ is externally hyperconvex in $B(x,r)$.
\item $A$ is externally hyperconvex in $B(A,s)$, for all $s>0$.
\item $A$ is externally hyperconvex in $B(A,s)$, for some $s>0$.
\item There is a proximinal 1-Lipschitz retraction $\rho \colon X \to A$, such that for all $x,y \in X$, one has  
\[
d(\rho(x),y) \leq \max\{d(x,y),d(A,y)\}.
\]
\end{enumerate}
Moreover, if $\sigma$ is a convex geodesic bicombing on $X$, the following is also equivalent:
\begin{enumerate}[(i)]
\setcounter{enumi}{6}
\item $A$ is $\sigma$-convex and uniformly locally externally hyperconvex in $X$.
\end{enumerate} 
\end{Thm}


\begin{Thm}\label{Thm:WeaklyExternallyHyperconvex}
Let $X$ be a hyperconvex metric space and $A \subset X$ a subset. Then the following are equivalent:
\begin{enumerate}[(i)]
\item $A$ is weakly externally hyperconvex in $X$.
\item For all $x \in X$ and for all $r \geq d(x,A)$, $A \cap B(x,r)$ is weakly externally hyperconvex in $B(x,r)$.
\item There is some $x \in X$ such that for all $r \geq d(x,A)$, $A \cap B(x,r)$ is weakly externally hyperconvex in $B(x,r)$.
\item $A$ is weakly externally hyperconvex in $B(A,s)$, for all $s>0$.
\item $A$ is weakly externally hyperconvex in $B(A,s)$, for some $s>0$.
\item $A$ is a proximinal 1-Lipschitz retract.
\end{enumerate}
Moreover, if $A$ possesses a convex geodesic bicombing, straight lines in $X$ are unique and $\sigma$ is a convex geodesic bicombing on $X$, the following is also equivalent:
\begin{enumerate}[(i)]
\setcounter{enumi}{6}
\item $A$ is $\sigma$-convex and uniformly locally weakly externally hyperconvex in $X$.
\end{enumerate} 
\end{Thm}

Note that the extra conditions are fulfilled if $X$ is a proper metric space with finite combinatorial dimension. 


\appendix

\section{The (n,k)-Intersection Property}

Lindenstrauss' proof of the $(F,k)$-intersection property for Banach spaces \cite[Theorem~4.1]{Lin} is essentially built on the existence of barycenters. But barycenters already exist in metric spaces with a geodesic bicombing. The following theorem is for instance stated in \cite{DesL2}. Below, $\mathfrak{S}_m$ denotes the symmetric group of order $m$.

\begin{Thm}\label{Thm:barycenter}
Let $(X,d)$ be a complete metric space with a geodesic bicombing~$\sig$. Then, for $m \in \mathbb{N}$, there exists a map $\bary_m \colon X^m \to X$ 
such that the following hold:
\begin{enumerate}[(i)]
\item
$\bary_m(x_1,\ldots,x_m)$ lies in the closed $\sig$-convex hull of 
$\{x_1,\ldots,x_m\}$;
\item
$d(\bary_m(x_1,\ldots,x_m),\bary_n(y_1,\ldots,y_m)) \le 
\min_{\pi \in \mathfrak{S}_m} \tfrac{1}{m} \sum_{i=1}^m d(x_i,y_{\pi(i)})$;
\item
$\varphi(\bary_m(x_1,\ldots,x_m)) = \bary_m(\varphi(x_1),\ldots,\varphi(x_m))$ whenever 
$\varphi$ is an isometry of $X$ and $\sig$ is such that for any  $(x,y) \in X \times X$ one has $\varphi \circ \sigma_{xy} = \sigma_{\varphi(x)\varphi(y)}$, i.e. $\sigma$ is $\varphi$-equivariant.
\end{enumerate}
We then call $\bary_m$ a barycenter map.
\end{Thm}

The construction satisfies $\bary_1(x) := x$ and $\bary_2(x,y) := \sig_{xy}(\tfrac{1}{2}) = \sig_{yx}(\tfrac{1}{2})$ as well as for $m \ge 3$,
\[
\bary_m(x_1,\ldots,x_m) = \bary_m(\bary_{m-1}(\hat{x}^1),\ldots,\bary_{m-1}(\hat{x}^m)),
\]
where $\hat{x}^i := (x_1,\ldots,x_{i-1},x_{i+1},\ldots,x_m)$.

We say that a metric space $X$ has the \emph{$(n,k)$-intersection property} (short $(n,k)$-IP) if for every family of $n$ closed balls such that any subfamily of $k$ balls has a non-empty intersection, the overall intersection is non-empty. If $X$ has the $(n,k)$-IP for every $n$, then we say that $X$ has the \emph{$(F,k)$-intersection property}. Note that this is the same as saying that the subset of closed balls is a Helly family of order $k$. This yields the following generalization of \cite[Theorem~4.1]{Lin}.

\begin{Thm}
	Let $n,k \in \mathbb{N}$ such that $k \geq 2$ and $$n > \frac{4k-5+\sqrt{1+8(k-1)^2}}{2}.$$
	Then, for any complete metric space with a geodesic bicombing, the $(n,k)$-intersection property implies the $(F,k)$-intersection property.
\end{Thm}

\begin{proof}
	It is enough to show that the $(n,k)$-IP implies the $(n+1,k)$-IP. Let $\{ B(x_i,r_i) \}_{i=1}^{n+1}$ be a collection of closed balls such that any $k$ of them have a non-empty intersection. For $J \subset I_{n+1} := \{ 1, \ldots ,n+1 \}$, define $B_J := \bigcap_{i \in J} B(x_i,r_i)$. Fix some arbitrary $x \in X$ and let $$R := \max_{J \subset I_{n+1}, |J|=k-1} d(x,B_J).$$ For $\epsilon>0$, the ball $B(x,(1+\epsilon)R)$ and any subfamily of $k-1$ balls intersect. Consider now $\Omega := \{ \alpha \subset I_{n+1} : |\alpha|=n-1 \}$ with $N:=|\Omega| = \frac{n(n+1)}{2}$. Then, by the $(n,k)$-IP, there is some $y_\alpha \in B_\alpha \cap B(x, (1+\epsilon)R)$, for every $\alpha \in \Omega$. Let $$y := \bary_N((y_\alpha)_{\alpha \in \Omega}) \in B(x,(1+\epsilon)R).$$	
	Consider now $J \subset I_{n+1}$ with $|J|=k-1$. For all $\alpha \in \Omega$, we choose $z_\alpha \in B_J \cap B(x,(1+\epsilon)R)$ such that either $z_\alpha = y_\alpha$ if $y_\alpha \in B_J \cap B(x,(1+\epsilon)R)$ or 
	\begin{align*}
	d(y_\alpha,z_\alpha)
	&\leq (1+\epsilon) d(y_\alpha, B_J \cap B(x,(1+\epsilon)R) \\
	&\leq (1+\epsilon)(d(y_\alpha,x) + d(x, B_J \cap B(x,(1+\epsilon)R)) \leq 2(1+\epsilon)^2R,
	\end{align*}
	otherwise. Define $z:= \bary_N((z_\alpha)_{\alpha \in \Omega}) \in B_J$. Note that $$N' := |\{\alpha \in \Omega : J \subset \alpha \}| = \frac{(n-k+1)(n-k+1)}{2}$$ and hence
	\begin{align*}
		d(y,B_J) \leq d(y,z) \leq \frac{1}{N}\sum_{\alpha \in \Omega} d(y_\alpha,z_\alpha) \leq \tfrac{N-N'}{N} \cdot 2(1+\epsilon)^2R,
	\end{align*}
	i.e. $d(y,B_J) \leq c R$ for $c=\frac{2(N-N')}{N} (1+\epsilon)^2$.
	Therefore, if $\frac{N-N'}{N} < \frac{1}{2}$, we have $c<1$ for $\epsilon$ small enough. Thus we can find a sequence $(y_j)_j$ with $d(y_j,y_{j+1})\leq c^j R$ and $d(y_j,B_J) \leq c^j R$. This is a Cauchy sequence with limit point $y_\infty \in \bigcap_{i=1}^{n+1}B(x_i,r_i)$.	
\end{proof}

\textbf{Acknowledgments.} The authors gratefully acknowledge support from the Swiss National Science Foundation.



\end{document}